\documentclass[a4paper,12pt]{amsart}
\usepackage{amssymb}
\usepackage{amsmath, amsthm, amscd, amsfonts, amssymb, graphicx, color}
\usepackage{amsmath, amsthm, amscd, amsfonts, amssymb, graphicx, color}

\newtheorem{theorem}{Theorem}[section]
\newtheorem{lemma}[theorem]{Lemma}

\newtheorem{corollary}[theorem]{Corollary}
\theoremstyle{definition}
\newtheorem{definition}[theorem]{Definition}
\newtheorem{example}[theorem]{Example}

\theoremstyle{remark}
\newtheorem{remark}[theorem]{Remark}
\numberwithin{equation}{section}

\title[Hypercyclic generalized shift operators]
{Hypercyclic generalized shift operators}

\author[S. Ivkovi\'{c}]{Stefan Ivkovi\'{c}}
\address{Mathematical Institute of the Serbian Academy of Sciences and Arts,
	p.p. 367, Kneza Mihaila 36, 11000 Beograd, Serbia}
\email{stefan.iv10@outlook.com }

\author[S. M. Tabatabaie]{Seyyed Mohammad Tabatabaie}
\address{Department of Mathematics, University of Qom, Qom, Iran}
\email{sm.tabatabaie@qom.ac.ir}



\date{\today}

\begin{document}

\maketitle

\begin{abstract}
In this paper, we study linear dynamical properties of shift operators on some classes of Segal algebras.  Moreover,  we characterize hypercyclic generalized bilateral shift operators on the standard Hilbert module.
\end{abstract}

\baselineskip17pt

\section{Introduction and Preliminaries}
Hypercyclicity, as a main linear dynamical property, is related to very notions of topological dynamics, such as topological transitivity, topological mixing, linear chaos and so on, and  has been an  active research topic in mathematics, fruitful results and theories appeared during the last four decades. We refer to these two monographs  \cite{bmbook,gpbook} on this theme.

Among other works, H. Salas in \cite{sa95} presented some characterization of hypercyclic weighted shifts on $\ell^p(\mathbb{Z})$. Then, inspired by this paper and using the aperiodic elements, some sufficient and necessary conditions were obtained in \cite{cc09,cc11} for weighted translation operators to be hypercyclic on the Lebesgue space in context of homogeneous spaces and locally compact groups; see also \cite{chta2,kum} for study of hypercyclicity in the context of hypergroups.

Recently, these works made a significant motivation for researchers to study other kinds of shift operators on various function spaces such as Orlicz spaces to demonstrate some cruel results and construct important examples; see \cite{ccot, cot,kalmes10,kostic}. Finally, chaotic and hypercyclic operators and related topics were studied on some class of solid Banach function spaces in \cite{tabsaw,chta3,tabiv2}, which cover so many kinds of known Banach function spaces. Recently, hypercyclic translation operators also have been studied on the algebra of compact operators in \cite{ivc}; see also \cite{pet}. 

In this paper, we give a characterization of generalized hypercyclic bilateral shift operators which are defined on the standard Hilbert module $\ell^2(\mathcal{A})$, where $\mathcal{A}$ is a proper ideal of a unital commutative C$^\ast$-algebra. Also, We charactrize some classes of hypercyclic operators defined on a  C$^*$-algebra $\mathcal{A}$ corresponding  to an isometric $\ast$-isomorphism $\Phi$ and a fixed element $b\in\mathcal{A}$.
In last section, by some technical lemmas, first we construct an approximate identity for Segal algebras $\mathcal{A}_\tau$ which were introduced in \cite{tak}. Then, as an application, we  give some sufficient and necessary conditions for a new class of shift operators on $\mathcal{A}_\tau$ to be hypercyclic.

Next, we recall some definitions and introduce some notations.
If $\mathcal X$ is a Banach space, the set of all bounded linear operators from $\mathcal X$ into $\mathcal X$ is denoted by $B(\mathcal X)$. Also, we denote $\mathbb{N}_0:=\mathbb{N}\cup\{0\}$.
\begin{definition}
	Let $\mathcal X$ be a Banach space. A sequence $(T_n)_{n\in \mathbb{N}_0}$  of operators in $B(\mathcal X)$ is called {\it topologically transitive} if for each non-empty open subsets $U,V$ of
	${\mathcal X}$, $T_n(U)\cap V\neq \varnothing$ for some $n\in \mathbb{N}$. If $T_n(U)\cap V\neq \varnothing$ holds from some $n$ onwards, then
	$(T_n)_{n\in \mathbb{N}_0}$ is called {\it topologically mixing}. 
\end{definition}
\begin{definition}
	Let $\mathcal X$ be a Banach space. A sequence $(T_n)_{n\in \mathbb{N}_0}$  of operators in $B(\mathcal X)$ is called {\it hypercyclic} if there is an element $x\in\mathcal X$ (called \emph{hypercyclic vector}) such that the orbit $\mathcal O_x:=\{T_nx:\,n\in\mathbb N_0\}$ is dense in $\mathcal X$. The set of all hypercyclic vectors of a sequence $(T_n)_{n\in \mathbb{N}_0}$ is denoted by $HC((T_n)_{n\in \mathbb{N}_0})$. If $HC((T_n)_{n\in \mathbb{N}_0})$ is dense in $\mathcal X$, the sequence $(T_n)_{n\in \mathbb{N}_0}$ is called \emph{densely hypercyclic}. An operator $T\in B(\mathcal X)$ is called \emph{hypercyclic} if the sequence $(T^n)_{n\in \mathbb N_0}$ is hypercyclic.
\end{definition}
Note that a sequence $(T_n)_{n\in \mathbb{N}_0}$  of operators in $B(\mathcal X)$ is topologically transitive if and only if it is densely hypercyclic \cite[Theorem 1.57]{gpbook}. In particular, an operator $T\in B(X)$ is hypercyclic if and only if it is topologically transitive \cite[Theorem 2.19]{gpbook}. A Banach space admits a hypercyclic operator if and only if it is separable and infinite-dimensional \cite{bg99}. So, in this paper we assume that Banach spaces are separable and infinite-dimensional.
\begin{definition}
	Let $\mathcal X$ be a Banach space, and $(T_n)_{n\in \mathbb{N}_0}$  be a sequence of operators in $B(\mathcal X)$. A vector $x\in \mathcal X$ is called a {\it periodic element} of $(T_n)_{n\in \mathbb{N}_0}$ if there exists a constant $N\in\mathbb N$ such that for each $k\in\mathbb N$, $T_{kN}x=x$. The set of all periodic elements of $(T_n)_{n\in \mathbb{N}_0}$ is denoted by
${\mathcal P}((T_n)_{n\in \mathbb{N}_0})$. The sequence $(T_n)_{n\in \mathbb{N}_0}$ is called {\it chaotic} if $(T_n)_{n\in \mathbb{N}_0}$ is topologically transitive and ${\mathcal P}((T_n)_{n\in \mathbb{N}_0})$ is dense in ${\mathcal X}$. An operator $T\in B(\mathcal{X})$ is called \emph{chaotic} if the sequence $\{T^n\}_{n\in \mathbb{N}_0}$ is chaotic. 
\end{definition}
\section{Generalized hypercyclic bilateral shift operators}
Let $\mathcal{A}$ be a commutative C$^\ast$-algebra with norm $\|\cdot\|$, and let  $\ell^2(\mathcal{A})$ be the standard Hilbert module over $\mathcal{A}$. We denote the norm of $\ell^2(\mathcal{A})$ by $\|\cdot\|_2$. Assume that $\Phi:\mathcal{A}\rightarrow\mathcal{A}$ is an isometric $\ast$-isomorphism. Fix a sequence $b:=(b_j)_{j\in\mathbb{Z}}\subseteq\mathcal{A}$  such that $\sup_{j\in\mathbb{Z}}\|b_j\|<\infty$. For each $n\in\mathbb{N}$ define the mapping 
$$T_{\Phi,b,n}:\ell^2(\mathcal{A})\rightarrow \ell^2(\mathcal{A}),\quad \big(T_{\Phi,b,n}(a)\big)_{j}:=\left[\prod_{k=1}^{n} \Phi^{n-k}(b_{j-n+k})\right] \Phi^{n}(a_{j-n}),$$
 where $a=(a_k)_{k\in\mathbb{Z}}\in\ell^2(\mathcal{A})$ and $j\in\mathbb{Z}$.

\begin{remark}
	For each $a:=(a_j)_j\in\mathcal\ell^2(\mathcal A)$ we have $\|a_j\|\leq \|a\|_2$ for all $j\in\mathbb Z$. Indeed,
	\begin{align*}
	\|a_j\|^2&=\sup\{|h(a_j)|^2:\,h\in\sigma(\mathcal A)\}\\
	&\leq \sup\big\{\sum_{i\in\mathbb Z}|h(a_i)|^2:\,h\in\sigma(\mathcal A)\big\}\\
	&=\sup\big\{|\sum_{i\in\mathbb Z}h(a_i^*\,a_i)|:\,h\in\sigma(\mathcal A)\big\}\\
	&=\sup\big\{|h(\sum_{i\in\mathbb Z}a_i^*\,a_i)|:\,h\in\sigma(\mathcal A)\big\}\\
	&=\|\sum_{i\in\mathbb Z} a_j^*\,a_j\|=\|a\|_2^2.
	\end{align*}
\end{remark}
Next, the spectrum of $\mathcal A$ is denoted by $\sigma(\mathcal A)$. Also, for each $J\in\mathbb{N}$ we denote $I_J:=[-J,J]\cap\mathbb Z$.
\begin{theorem}\label{thm1}
We have (1)$\Rightarrow$(2):
\begin{enumerate}
\item The sequence $(T_{\Phi,b,n})_{n\in\mathbb{N}_0}$ is hypercyclic. 
\item For every non-empty compact subset $K$ of $\sigma(\mathcal A)$ and each non-empty finite subset $I$ of $\mathbb{N}$ there exists a strictly increasing sequence $(r_{k})_{k \in \mathbb{N}} $ in $\mathbb{N}$ such that 
	$$ \lim_{k\rightarrow\infty} \left(\sup_{h\in K} \prod_{i=1}^{r_{k}} \left|h\big(\Phi^{-i}(b_{i+j})\big)\right|\right)=\lim_{k\rightarrow\infty} \left(\sup_{h\in K} \prod_{i=1}^{r_{k}} \left|h\big(\Phi^{i-1}(b_{j-i+1})\big) \right|^{-1}\right)=0$$ 
	for all $j\in I$.
	\end{enumerate}
Also, if $\mathcal{A}$ is unital and for each $j$, $b_j$ is invertible in $\mathcal{A}$ with $\sup_{j\in\mathbb{Z}}\|b_j^{-1}\|<\infty$, then (2)$\Rightarrow$(1). 
\end{theorem}
\begin{proof}
	Assume that $(T_{\Phi,b,n})_{n\in\mathbb{N}_0}$ is hypercyclic. Note that since $T_{\Phi,b,n}=T_{\Phi,b,1}^n$ for all $n$, the sequence $(T_{\Phi,b,n})_{n\in\mathbb{N}_0}$ is densely hypercyclic. Let $ \varnothing\neq I\subset\mathbb{N}$ be finite and $\varnothing\neq K\subseteq \sigma(\mathcal A)$ be compact. 
	Then, there is some $y\in\mathcal A$ such that $h(y)=1$ for all $h\in K$.
Define an element  $z:=(z_j)_{j\in\mathbb{Z}}\in\ell^2(\mathcal A)$ with $z_j:=y$ if $j\in I$ and $z_j:=0$ otherwise. So, there is a hypercyclic vector $a^{(1)} \in  \ell^2(\mathcal{A})$ for  the sequence $(T_{\Phi,b,n})_{n\in\mathbb{N}_0}$ such that 
\begin{equation}\label{eq0}
\|a^{(1)}-z\|_2<\frac{1}{2}.
\end{equation}
 Hence, setting $J:=\max_{j\in I}|j|$ we can find an $r_{1}>2J$ such that 
\begin{equation}\label{eq1}
\|T_{\Phi,b,r_1}(a^{(1)})-z\|_2 < \frac{1}{4}.
\end{equation}

 This implies that for all $j\in I$ we have $\|a_{j}^{(1)}-y\|<\frac{1}{2}$, and therefore 
\begin{equation}\label{eq2}
|h(a_j^{(1)})|>\frac{1}{2},\qquad (h\in K).
\end{equation}
Next, by \eqref{eq1},
	$$\Big\|\left[\prod_{i=1}^{r_{1}} \Phi^{r_1-i}(b_{j+i})\right]\,\Phi^{r_1}(a_{j}^{(1)})\Big\|=\Big\|\left[\prod_{i=1}^{r_{1}} \Phi^{r_1-i}(b_{j+i})\right]\,\Phi^{r_1}(a_{j}^{(1)})-z_{j+r_{1}}\Big\| <\dfrac{1}{4}$$ 
	for all $j \in I$, because $r_{1} > 2J$ implies that $z_{j+r_{1}}=0$.
	Since $\Phi$ is an isometric $\ast$-isomorphism, we get 
	$$\Big\|\left[\prod_{i=1}^{r_{1}} \Phi^{-i}(b_{j+i})\right] \,a_{j}^{(1)}\Big\|< \dfrac{1}{4}$$ 
	for all $j\in I$, and hence 
		$$\left[\prod_{i=1}^{r_{1}} |h(\Phi^{-i}(b_{j+i}))|\,\right] \,|h(a_{j}^{(1)})|< \dfrac{1}{4}$$
	for all $h\in K$ and $j \in I$.
	By \eqref{eq2}, this follows that 
			$$\prod_{i=1}^{r_{1}} |h(\Phi^{-i}(b_{j+i}))|< \dfrac{1}{2}$$
	for all $h\in K$ and $j \in I$.
 Similarly, for each $k \in \mathbb{N}$ we can find a hypercyclic vector $ a^{(k)} \in \ell^2(\mathcal A)$ for $(T_{\Phi,b,n})_{n\in\mathbb{N}_0}$ such that
  $\|a^{(k)}-z\|<\frac{1}{2^{k}}$. Moreover, we can choose $r_{k} >r_{k-1}> \cdots > r_{1} >2J$ such that  
  \begin{equation}\label{eq3}
  \|T_{\Phi,b,r_k}(a^{(k)})-z\|_2 < \frac{1}{2^{2k}}.
  \end{equation}
   By proceeding as above,
   			$$\prod_{i=1}^{r_{k}} |h(\Phi^{-i}(b_{j+i}))|< \dfrac{1}{2^k(2^k-1)}$$
   for all $h\in K$ and $j \in I$, and hence
	$$\sup_{h\in K} \prod_{i=1}^{r_{k}} |h(\Phi^{-i}(b_{j+i}))|\leq \frac{1}{2^{k}(2^k-1)}.$$ 
	for all $ k \in \mathbb{N}$. This implies that  
	$$ \lim_{k\rightarrow\infty}\left(\sup_{h\in K} \prod_{i=1}^{r_{k}} |h(\Phi^{-i}(b_{j+i}))|\right)=0.$$ 
	for all $j\in I$.

	Now, the relation \eqref{eq0} implies that 
	\begin{equation}\label{eq4}
	\|a_{j-r_{1}}^{(1)} \| < \frac{1}{2}
	\end{equation}
	for all $j \in I$ because  $z_{j-r_{1}}=0 $ by the choice of $z$. Also, by \eqref{eq1}, 
$$\Big\| \left[\prod_{i=1}^{r_{1}} \Phi^{r_1-i}(b_{j-r_{1}+i})\right]\,\Phi^{r_1}(a_{j-r_{1}}^{(1)})-y \Big\|<\frac{1}{4}$$ 
for all $j\in I$, so  for all $h\in K$ we have
$$\left|\left[\prod_{i=1}^{r_{1}} h(\Phi^{r_1-i}(b_{j-r_{1}+i}))\right]\,h(\Phi^{r_1}(a_{j-r_{1}}^{(1)}))-1\right|< \dfrac{1}{4}.$$
and this implies that
$$\left|\left[\prod_{i=1}^{r_{1}} h(\Phi^{r_1-i}(b_{j-r_{1}+i}))\right]\,h(\Phi^{r_1}(a_{j-r_{1}}^{(1)}))\right| > \frac{3}{4}.$$
Now, thanks to the inequality \eqref{eq4} we get 
$$\prod_{i=1}^{r_{1}} \left|h(\Phi^{r_1-i}(b_{j-r_{1}+i}))\right|^{-1}<\frac{2}{3}$$ 
for all $h\in K$ and $j \in I$. Hence, substituting $s=r_{1}-i+1 ,$ we get that 
$$\sup_{h\in K} \prod_{s=1}^{r_{1}} \left|h(\Phi^{s-1}(b_{j+1-s}))\right|^{-1} < \frac{2}{3}$$ 
for all $j \in I$. Similarly, using that $\|a^{(k)}-z\|_2 <\dfrac{1}{2^{k}} $ and \eqref{eq3} we conclude that
$$\sup_{h\in K} \prod_{s=1}^{r_k} \left|h(\Phi^{s-1}(b_{j+1-s}))\right|^{-1}<  \dfrac{1}{2^{k}-\frac{1}{2^{k}}}  $$ for all $k \in \mathbb{N} $ and $j\in I$, and this completes the proof.

Conversely, assume that (2) holds. Let $\mathcal{F}$ be the set of  all sequences $(y_j)\in\ell^2(\mathcal{A})$ such that for some finite subset $I$ of $\mathbb{Z}$, $y_j=0$ for all $j\in\mathbb{Z}\setminus I$, and ${\rm supp}\widehat{y_{j}}$ is compact for all $j\in I$, where $\widehat{y_{j}}$ is the Gelfand transform of $y_j$. Then, the set $\mathcal{F}$
is dense in $\ell^2(\mathcal{A})$. Let $y:=(y_j)_j \in \mathcal{F}$. Then, there exists some $J\in\mathbb{N}$ such that $y_j=0$ for all $j\in\mathbb{Z}\setminus I_J$, and   $K_j:={\rm supp}\widehat{y_{j}}$  is compact for all $ j \in I_J$. Put $K:=\bigcup_{j=-J}^JK_j$. By the assumption, there exists an increasing sequence $(r_{k})_{k \in \mathbb{N}} $ in $ \mathbb{N}$ such that 	
$$ \lim_{k\rightarrow\infty} \left(\sup_{h\in K} \prod_{i=1}^{r_{k}} \left|h\big(\Phi^{-i}(b_{i+j})\big)\right|\right)=\lim_{k\rightarrow\infty} \left(\sup_{h\in K} \prod_{i=1}^{r_{k}} \left|h\big(\Phi^{i-1}(b_{j-i+1})\big) \right|^{-1}\right)=0.$$
For each and $j \in I_J$ and $k \in \mathbb{N}$ we have
$$(T_{\Phi,b,r_k}(y))_{j+r_{k}}=  \prod_{i=1}^{r_{k}} \Phi^{r_{k}-i}(b_{j+i}) \Phi^{r_k}(y_{j}),$$ 
and so, since $\Phi$ is an isometric $\ast$-isomorphism, 
$$\|(T_{\Phi,b,r_k}(y))_{j+r_{k}}\|=\|\prod_{i=1}^{r_{k}} \Phi^{-i}(b_{j+i}) \,y_{j}\|.$$ 
However, for each $j\in I_J$ we have
$$\|\prod_{i=1}^{r_{k}} \Phi^{-i}(b_{j+i}) \,y_{j}\| \leq  \sup_{h\in K} \left[\prod_{i=1}^{r_{k}} |h(\Phi^{-i}(b_{j+i})|\right]\,\| y_{j} \|$$ 
since ${\rm supp}\widehat{y_j}\subseteq K$. This follows that 
$$\lim_{k \rightarrow \infty} \| (T_{\Phi,b,r_k}(y))_{j+r_{k}} \|=0 $$ 
for all $j \in I_J$. As $y_{j}=0 $ whenever $ j \notin I_J$, we have 

$$(T_{\Phi,b,r_k}(y))_{j+r_{k}} =  \prod_{i=1}^{r_{k}} \Phi^{r_{k}-i}(b_{j+i}) \,\Phi^{r_k}(y_{j})=0 $$
for all $j \in \mathbb{Z} \setminus I_J$. It is easy to see that $\lim_{k \rightarrow \infty}T_{\Phi,b,r_k}(y) =0 $ in $\ell^2(\mathcal{A})$.  Now,  define $S_{\Phi,b,1}:\ell^2(\mathcal{A})\rightarrow\ell^2(\mathcal{A})$ by
$(S_{\Phi,b,1})_j:=\Phi^{-1}(b_{j+1}^{-1} a_{j+1})$, and put $S_{\Phi,b,n}:=S_{\Phi,b,1}^n$ for all $n\in\mathbb{N}$. Then, 
 for all $ j\in\mathbb{Z}$ and $k \in \mathbb{N}$ we have 
$$(S_{\Phi,b,r_k}(y))_{j+r_{k}} =  \prod_{i=1}^{r_{k}} \Phi^{i-r_{k}-1}(b_{j+1-i}^{-1}) \Phi^{-r_{k}}\,(y_{j}).$$ 
Hence, for each $j\in I_J$,
\begin{align*}
\| (S_{\Phi,b,r_k}(y))_{j+r_{k}} \|&=\left\| \prod_{i=1}^{r_{k}} \Phi^{i-1}(b_{j+1-i}^{-1} ) \,\,y_{j}\right\|\\
&\leq 
\sup_{h\in K} \left[\prod_{i=1}^{r_{k}} |h(\Phi^{i-1}(b_{j+1-i}^{-1}))|\right]\, \| y_{j} \|.
\end{align*}
Therefore, 
$$\lim_{k\rightarrow\infty} \| (S_{\Phi,b,r_k}(y))_{j+r_{k}}\|=0 $$
for all $j \in I_J$. Similarly, as above we deduce that $\lim_{k\rightarrow\infty}S_{\Phi,b,r_k}(y)=0$, so $T_{\Phi,b,r_k}$ and $S_{\Phi,b,r_k}$ converge pointwise on $\mathcal{F}$ which is dense in $ \ell^2(\mathcal{A}).$ Hence, we deduce that $T_{\Phi,b,r_k}$ is hypercyclic on $\ell^2(\mathcal{A})$.
\end{proof}
\begin{remark}\label{rem1}
	We mention that in the above discussion, if  $\mathcal A$ is a closed proper ideal of a unital commutative C$^\ast$-algebra $\mathcal E$ such that the characters of $\mathcal{A}$ can be  extended to characters on $\mathcal{E}$, and if  $\Phi:\mathcal{E}\rightarrow\mathcal{E}$ is an isometric $\ast$-isomorphism with $\Phi(\mathcal{A})=\mathcal{A}$, one can pick the sequence $b:=(b_j)_{j\in\mathbb{Z}}\subseteq\mathcal{E}$. In this case, some fact similar to Theorem \ref{thm1} holds without the unital condition at the second part. In particular, one can consider $\mathcal{E}:=C_b(X)$ and $\mathcal{A}:=C_0(X)$, where $X$ is a locally compact non-compact Hausdorff space. Another example is $\mathcal{E}:= \text{the classical unitalization of } \mathcal{A}$.
\end{remark}
Let $X$ be a locally compact, non compact Hausdorff space. Let $W:=(w_{j})_{j \in \mathbb{Z}} $ be a sequence in $C_b(X) $ with $\sup_j\|w_{j}\|_{\infty}<\infty$.  Let $\alpha $ be a homeomorphism on $X$, and define 
$$\Phi_\alpha:C_b(X)\rightarrow C_b(X),\quad \Phi_\alpha(f)(x):=f(\alpha(x))$$
for all $f\in C_b(X)$ and $x\in X$. Then, $\Phi_\alpha$ is an isometric $\ast$-isomorphism on $C_b(X)$. We denote $T_{\alpha , W}:=T_{\Phi_\alpha,W,1}$.
The next result is a direct conclusion of Theorem \ref{thm1} and Remark \ref{rem1}.
\begin{corollary}\label{cor0}
	Let $X$ be a locally compact Hausdorff space. Let $W:=(w_{j})_{j \in \mathbb{Z}} $ be a sequence in $C_{b}(X) $ with $\sup_j\|w_{j}\|_{\infty}<\infty$.  Let $\alpha $ be a homeomorphism on $X$. Then, we have (1)$\Rightarrow$(2):
\begin{enumerate}
	\item $ T_{\alpha, W} $ is hypercyclic.
	\item  For every non-empty compact subset $K$ of $X$ and each non-empty finite subset $I$ of $\mathbb{N}$ there exists a strictly increasing sequence $\{ r_{k} \}_{k \in \mathbb{N}} $ in $\mathbb{N}$ such that 
	$$ \lim_{k\rightarrow\infty} (  \sup_{t\in K} \prod_{i=1}^{r_{k}} |( w_{i+j} \circ \alpha^{-i}) (t) |  ) = \lim_{k\rightarrow\infty} (  \sup_{t\in K}  \prod_{i=1}^{r_{k}} |( w_{j+1-i} \circ \alpha^{i-1})^{-1} (t) |=0   $$ 
	for every $j \in I$.
\end{enumerate}
Also, if for each $j$, $w_j> 0$ with $\sup_{j\in\mathbb{Z}}\|w_j^{-1}\|_\infty<\infty$, then (2)$\Rightarrow$(1). 
\end{corollary}
\begin{example}\label{ex1}
	Let $X:=\mathbb{R}$ and $M>1$. Pick an $\epsilon>0$ such that $1+\epsilon<M$ and $1-\epsilon>\frac{1}{M}$. Let $\alpha:\mathbb{R}\rightarrow\mathbb{R}$ is defined by $\alpha(t):=t-1$. Let $(w_j)_{j\in\mathbb{Z}}$ be a sequence in $C_b(\mathbb{R})$ such that:
	\begin{itemize}
		\item $\|w_j\|_\infty\|w_j^{-1}\|_\infty<M$ for all $j\in\mathbb{Z}$,
		\item $|w_j(t)|\leq 1-\epsilon$ for all $j\in\mathbb{N}$ and $t\geq 0$,
		\item $|w_j(t)|\geq 1+\epsilon$ for all $j\in \mathbb{Z}\setminus\mathbb{N}$ and $t\leq 0$.
	\end{itemize}
	Then, the homeomorphism $\alpha$ and the sequence $(w_j)$ satisfy the conditions of Corollary \ref{cor0}.
\end{example}
Similar to the proof of Theorem \ref{thm1}, one can prove the following result.
\begin{theorem}\label{thm11}
		If $\mathcal{A}$ is unital and for each $j$, $b_j$ is invertible in $\mathcal{A}$ with $\sup_{j\in\mathbb{Z}}\|b_j^{-1}\|<\infty$, then (2)$\Rightarrow$(1):
	\begin{enumerate}
		\item The sequence $(T_{\Phi,b,n})_{n\in\mathbb{N}_0}$ is topologically mixing. 
		\item For every non-empty compact subset $K$ of $\sigma(\mathcal A)$ and each non-empty finite subset $I$ of $\mathbb{N}$,
		$$ \lim_{n\rightarrow\infty} \left(\sup_{h\in K} \prod_{i=1}^{n} \left|h\big(\Phi^{-i}(b_{i+j})\big)\right|\right)=\lim_{n\rightarrow\infty} \left(\sup_{h\in K} \prod_{i=1}^{n} \left|h\big(\Phi^{i-1}(b_{j-i+1})\big) \right|^{-1}\right)=0$$ 
		for all $j\in I$.
	\end{enumerate} 
\end{theorem}
Let $\mathcal{E}$ be a unital commutative C$^\ast$-algebra with norm $\|\cdot\|$, and  $\mathcal A$ be a closed proper ideal of $\mathcal E$ such that the characters of $\mathcal A$ can be extended to characters on $\mathcal{E}$. Assume that $\Phi:\mathcal{E}\rightarrow\mathcal{E}$ is an isometric $\ast$-isomorphism such that $\Phi(\mathcal{A})=\mathcal{A}$.  Fix an element $b\in \mathcal{E}$  and define 
$$U_{\Phi,b}:\mathcal{A}\rightarrow\mathcal{A},\qquad U_{\Phi,b}(a):=b\,\Phi(a)$$
for all $a\in\mathcal{A}$. Then, $U_{\Phi,b}$ is a bounded linear operator.
Next, we characterize some classes of hypercyclic $U_{\Phi,b}$ operators.
\begin{theorem}\label{thm2}
	Under the above notations, let $b\in \mathcal{E}$ be invertible. Assume that for each compact subset $K$ of $\sigma(\mathcal{A})$ there is some $N>0$ such that 
	$$\{\gamma\circ\Phi^n:\,\gamma\in K,\,n\geq N\}\bigcap K=\varnothing.$$
	Then, the followings are equivalent:
	\begin{enumerate}
		\item  $U_{\Phi,b}$ is hypercyclic on $\mathcal{A}$.
		\item  For every compact set $K\subseteq\sigma(\mathcal{A})$ there exists a strictly increasing sequence $(n_{k})_{k} \subseteq \mathbb{N}$ such that  
		$$\lim_{k \rightarrow \infty} (   \sup_{\gamma\in K}\prod_{j=0}^{n_{k}-1}  |\gamma(\Phi^{j-n_{k}}(b) )|) = 
		\lim_{k \rightarrow \infty} (   \sup_{\gamma\in K}\prod_{j=0}^{n_{k}-1}  |\gamma(\Phi^{j}(b))|^{-1}) =0.$$
	\end{enumerate}
\end{theorem}
\begin{proof}
	Assume that $(1)$ holds and let $K \subseteq\sigma(\mathcal{A})$ be compact. Then, there exists some element $u \in \mathcal{A}$ such that ${\rm supp}u$ is a compact subset of $\sigma(\mathcal{A})$, and $\hat{u}=1$ on $K$. Since $U_{\Phi,b}$ is hypercyclic on $\mathcal{A},$ there exists a  hypercyclic vector $v$ for $U_{\Phi,b}$ such that $ \| v-u \|< \frac{1}{2}$. Next, by the assumption, we can find an $N \in \mathbb{N}$ such that
	\begin{equation}\label{eq11}
	\{\gamma\circ\Phi^n:\,\gamma\in {\rm supp}\hat{u},\,n\geq N\}\bigcap {\rm supp}\hat{u}=\varnothing.
	\end{equation}
	Since $v$ is a hypercyclic vector for $U_{\Phi,b}$,  we can find some $n_{1} \geq N$ such that 
	$$\| \left[\prod_{j=0}^{n_{1}-1} \Phi^{j}(b)\right]\, \Phi^{n_{1}}(v) - u \|=\| U_{\Phi,b}^{n_{1}} (v) - u \|< \frac{1}{4}.$$
	Hence 
	
	$$ \| \left[\prod_{j=0}^{n_{1}-1} \Phi^{j-n_{1}}(b)\right]\,v-\Phi^{-n_{1}}(u)\|<\frac{1}{4},$$
	so	
	$$\sup_{\gamma\in K} \left|\left[\prod_{j=0}^{n_{1}-1} \gamma(\Phi^{j-n_{1}}(b))\right]\,\gamma(v)-\gamma(\Phi^{-n_{1}}(u))\right|<\frac{1}{4}.$$
	
	However, $\gamma(\Phi^{-n_{1}}(u))=0$ for all $\gamma\in K$ because of \eqref{eq11}. Hence 	
	$$\sup_{\gamma\in K} \left|\left[\prod_{j=0}^{n_{1}-1} \gamma(\Phi^{j-n_{1}}(b))\right]\,\gamma(v)\right|<\frac{1}{4}.$$
	However, since $\|v-u\|<\frac{1}{2}$ and $\hat{u}=1$ on $K$, we must have that $|\gamma(v)|>\frac{1}{2}$ for all $\gamma\in K$. Hence 
	
	$$\sup_{\gamma\in K} \left|\prod_{j=0}^{n_{1}-1} \gamma(\Phi^{j-n_{1}}(b))\right|<\frac{1}{2}.$$
	Proceeding inductively, for each $k \in \mathbb{N}$ one can find a hypercyclic vector $v^{(k)}$ such that $ \|v^{(k)} - u \|<\frac{1}{2^{k}}$ and some $n_{k}>n_{k-1}>\cdots>n_{1}> N$ such that 
	$$\|U_{\Phi,b}^{n_{k}} (g^{(k)}) - u \|<\frac{1}{4^{k}}.$$
	By the above arguments we deduce that 
	$$\sup_{\gamma\in K} \left|\prod_{j=0}^{n_{k}-1} \gamma(\Phi^{j-n_{k}}(b))\right|<\frac{1}{2^k}.$$
	For each $n\in\mathbb{N}$ we denote 
	$$E_n:=\{\gamma\circ\Phi^n:\,\gamma\in K\}.$$
	We have 
	$$\sup_{\gamma\in E_{n_{1}}} \left|\left[\prod_{j=0}^{n_{1}-1} \gamma(\Phi^{j-n_{1}}(b))\right]\,\gamma(v)-\gamma(\Phi^{-n_{1}}(u))\right|<\frac{1}{4}.$$
	Hence
	$$\left|\left[\prod_{j=0}^{n_{1}-1} \gamma(\Phi^{j-n_{1}}(b))\right]\,\gamma(v)\right| > \frac{3}{4} $$
	for all $\gamma \in E_{n_{1}}$, since $\hat{u}=1$ on $K$. However,
	$$\sup_{\gamma \in E_{n_{1}}} |\gamma(v)|<\frac{1}{2}$$
	because $E_{n_1}\cap {\rm supp}\hat{u}=\varnothing$.
	Hence, we get 
	$$\sup_{\gamma\in E_{n_{1}}}\prod_{j=0}^{n_{1}-1} |\gamma(\Phi^{j-n_{1}}(b))|^{-1}<\frac{2}{3},$$
	which gives
	$$\sup_{\gamma\in K}\prod_{j=0}^{n_{1}-1} |\gamma(\Phi^{j}(b))|^{-1}<\frac{2}{3}.$$
	
	Proceeding indirectively as above, we can construct a strictly increasing sequence $(n_{k})\subseteq \mathbb{N}$ such that 
	$$\sup_{\gamma\in K}\prod_{j=0}^{n_{k}-1} |\gamma(\Phi^{j}(b))|^{-1}< \frac{2^{k}}{2^{2k}-1} <  \frac{1}{2^{k}-1}.$$
	Assume now that $(2)$ holds and let $v\in \mathcal{A}$ such that ${\rm supp}\hat{v}$ is compact. Set $K:={\rm supp}\hat{v}$. Then,
	\begin{align*}
	\| U_{\Phi,b}^{n_{k}} (v)  \|&=\|\prod_{j=0}^{n_{k}-1}\Phi^{j}(b) \Phi^{n_{k}}(v)\|\\
	&=\|\prod_{j=0}^{n_{k}-1} \Phi^{j-n_{k}}(b)\,v\|\\
	&\leq \sup_{\gamma\in K} \left[\prod_{j=0}^{n_{k}-1} |\gamma(\Phi^{j-n_{k}}(b))|\right]\,\, \|v\|.
	\end{align*}
	Define $V_{\Phi,b}(a):=\Phi^{-1}(b^{-1}a)$ for all $a\in\mathcal{A}$. Then, $V_{\Phi,b}$ is the inverse of $U_{\Phi,b}$ and 
	\begin{align*}
	\| V_{\Phi,b}^{n_{k}}(v)\|&= \|\prod_{j=1}^{n_{k}}\Phi^{-j}(b^{-1}) \Phi^{-n_{k}}(v)) \|\\
	&=\|   \prod_{j=1}^{n_{k}} \Phi^{n_{k}-j}(b^{-1})\,\,\, v\|\\
	&=\|   \prod_{j=0}^{n_{k}-1} \Phi^{j}(b^{-1})\,\,\,v\|\\
	&\leq \sup_{\gamma\in K} \left[\prod_{j=0}^{n_{k}-1} |\gamma(\Phi^{j}(b^{-1}) )|\right]\,\,\|v\|.
	\end{align*}
	So, $U_{\Phi,b}$ and $V_{\Phi,b}$ are hypercyclic.
\end{proof}
Let $X$ be a topological space, and $W:=(w_{j})_{j \in \mathbb{Z}} $ be a sequence in $C_b(X) $ with $\sup_{j}\|w_{j}\|_{\infty}<\infty$.  Let $\alpha $ be a homeomorphism on $X$. We denote $U_{\alpha , W}:=U_{\Phi_\alpha,W}$ as an operator on $C_0(X)$. Now, we conclude the next fact directly from Theorem \ref{thm2}.
\begin{corollary}\label{cor2}
	Let $X$ be a locally compact non-compact Hausdorff space, and $W:=(w_{j})_{j \in \mathbb{Z}} $ be a sequence in $C_b(X) $ such that $w_j>0$ for all $j$, and  
	$$\sup_{j}\|w_j^{-1}\|_\infty+\sup_{j}\|w_{j}\|_{\infty}<\infty.$$
	Also, assume that for every compact subset $K$ of $X$ there exists some $N > 0$ such that $\alpha^{n} (K) \cap K =\varnothing $ for all $n \geq N$.
	Then ,the followings are equivalent:
	\begin{enumerate}
	\item $U_{\alpha,w}$ is hypercyclic on $C_0(X)$.
	\item For every compact subset $K$ of $X$ there exists a strictly increasing sequence $(n_{k}) \subseteq \mathbb{N}$ such that  
	$$\lim_{k \rightarrow \infty} (\sup_{t \in K}\prod_{j=0}^{n_{k}-1}  |(w \circ \alpha^{j-n_{k}})(t)|)=\lim_{k \rightarrow \infty} (\sup_{t \in K} \prod_{j=0}^{n_{k}-1}|(w \circ \alpha^{j})^{-1}(t)|)=0.$$
\end{enumerate}
\end{corollary}
\begin{example}
	Let $G$ be a locally compact non-compact second countable group. Then, $G$ contains an aperiodic element. Hence, by \cite{cc11}, the mapping $\alpha_a:G\rightarrow G$ defined by $\alpha_a(x):=ax$ satisfies the condition of Corollary \ref{cor2}.
\end{example}
\begin{example}
	Let $X:=\mathbb{R}$ and assume that $M,\epsilon,\alpha$ are as in Example \ref{ex1}. If $w\in C_b(\mathbb{R})$ such that for some constant $k>0$, 
	we have $1+\epsilon\leq |w(t)|\leq M$ for all $t\leq -k$, and $\frac{1}{M}\leq |w(t)|\leq 1-\epsilon$ for all $t\geq k$. Then, the conditions of Corollary \ref{cor2} are satisfied.
\end{example}
Similar to Theorem \ref{thm2} the next result can be obtained.
\begin{theorem}\label{thm22}
Under the assumptions of Theorem \ref{thm2}, we have (2)$\Rightarrow$(1):
	\begin{enumerate}
		\item  $U_{\Phi,b}$ is topologically mixing on $\mathcal{A}$.
		\item  For every compact set $K\subseteq\sigma(\mathcal{A})$,
		$$\lim_{n\rightarrow \infty} (\sup_{\gamma\in K}\prod_{j=0}^{n-1}  |\gamma(\Phi^{j-n}(b) )|)= 
		\lim_{n \rightarrow \infty} (   \sup_{\gamma\in K}\prod_{j=0}^{n-1}  |\gamma(\Phi^{j}(b))|^{-1}) =0.$$
	\end{enumerate}
\end{theorem}
\section{Hypercyclic operators on Segal algebras}

In sequel, we assume that $\mathcal{A} =C_{0}(\mathbb{R})$ and $ \tau \in C_{b}(\mathbb{R})$. Put
$$\mathcal{A}_\tau:=\{f\in\mathcal{A}:\, \sum_{k=0}^\infty\|f\tau^k\|_{\infty} <\infty\}.$$
For each $f\in \mathcal{A}_\tau$ we define
$$\|f\|_\tau:=\sum_{k=0}^\infty\|f\tau^k\|_{\infty}.$$
Then, $\mathcal{A}_\tau$ is a Banach algebra \cite{tak}. We will call this algebra \emph{Segal algebra corresponding to $\tau$.}


For a homeomorphism $\alpha $ of $\mathbb{R} ,$ we consider corresponding Segal algebra $\mathcal{A}_{\tau \circ \alpha^{k}}$ for all $k \in \mathbb{Z}$. 
Clearly, $f \cdot (\tau \circ \alpha^{k}) \in \mathcal{A} $ for all $k \in \mathbb{Z} $ whenever $f \in C_{0}(\mathbb{R}) $ since $\tau \circ\alpha^{k}  \in C_{b}(\mathbb{R}) $ for all $  k \in \mathbb{Z} .$ 
By definition, $\| f \|_{\infty} \leq \|f \|_{\tau \circ \alpha^{k}} $ whenever $ f \in \mathcal{A}_{\tau \circ \alpha^{k}} .$

\begin{lemma}\label{lem1}
	If $f \in \mathcal{A}_{\tau} ,$ then $ | f |^{-1} ((0, \infty)) \cap | \tau |^{-1}  ( \left[1,\infty \right)  )=\varnothing.$ 
\end{lemma}	

\begin{proof}
	For each $x \in \mathbb{R} ,$ we have 
	$$| f(x) | \sum_{k =1}^{\infty}  | \tau(x) |^{k}=\sum_{k =1}^{\infty} | f(x) | \text{ } | \tau(x) |^{k} 
	\leq \sum_{k =1}^{\infty}  \| f\tau^{k} \|_{\infty} =  \| f \|_{\tau} < \infty.$$ 
	Hence, if $| f(x) |  >0 ,$ then the series $\sum_{k =1}^{\infty}  | \tau(x) |^{k}  $ is convergent. 
\end{proof}

\begin{lemma}
	Let $\epsilon_{1}, \epsilon_{2} \in (0,1) $ with $\epsilon_{2} < \epsilon_{1} .$ Then for every compact subset $K_{\epsilon_{2}}^{(\tau)} \subseteq | \tau |^{-1} ( \left[ 0,\epsilon_{2} \right]) ,$ there exists a function $\mu_{K_{\epsilon_{2}}^{(\tau)}, \epsilon_{1} }  \in \mathcal{A}_{\tau}$ such  that the followings hold:
	\begin{enumerate}
	\item  $\mu_{K_{\epsilon_{2}}^{(\tau)}, \epsilon_{1} } =1$ on $K_{\epsilon_{2}}^{(\tau)}$,
	 
\item  $\mu_{K_{\epsilon_{2}}^{(\tau)}, \epsilon_{1} } \in C_{c}(\mathbb{R})  ,$
	 
	\item  ${\rm supp}\,\mu_{K_{\epsilon_{2}}^{(\tau)}, \epsilon_{1} }  \subseteq | \tau |^{-1} ( \left[ 0,\epsilon_{1} \right])) .$
	\end{enumerate}
\end{lemma}

\begin{proof}
	Since $K_{\epsilon_{2}}^{(\tau)} $ is compact, there exists some $N \in \mathbb{N} $ such that $K_{\epsilon_{2}}^{(\tau)} \subseteq  \left[ -N,N \right] .$ Then 
	$$K_{\epsilon_{2}}^{(\tau)} \cap ( \text{ } |\tau|^{-1} ( \left[ \epsilon_{1},\infty \right) )     \cup   \left( - \infty,-N-1 \right]   \cup  \left[ N+1,\infty \right)  \text{ })  = \varnothing ,$$ 
	as $ K_{\epsilon_{2}}^{(\tau)} \subseteq |\tau|^{-1} (\left[ 0,\epsilon_{2} \right])  \cap \left[ -N,N \right]  .$ Since 
	$|\tau|^{-1} ( \left[ \epsilon_{1},\infty \right] ) \cup \left( - \infty,-N-1 \right]   \cup  \left[ N+1,\infty \right)  $  is closed, by Urisohn's lemma, there exists a continuous function $\mu_{K_{\epsilon_{2}}^{(\tau)}, \epsilon_{1} }^{(\tau)} $ such that 
	
	$$\mu_{ K_{\epsilon_{2}}^{(\tau)}, \epsilon_{1} }=1\,\text{ on } K_{\epsilon_{2}}^{(\tau)}, \quad 0 \leq  \mu_{  { K_{\epsilon_{2}}^{(\tau)}, \epsilon_{1} }  }^{(\tau)} \leq 1 \,\,\text{ and }$$ 
	
	$$ \mu_{K_{\epsilon_{2}}^{(\tau)}, \epsilon_{1}}=0\,\text{ on } |\tau|^{-1} ( \left[ \epsilon_{1},\infty \right] ) \cup \left( - \infty,-N-1 \right]   \cup[ N+1,\infty).$$
	It follows that ${\rm supp}\,\mu_{K_{\epsilon_{2}}^{(\tau)}, \epsilon_{1} }\subseteq \left[ -N-1,N+1 \right] ,$ so $\mu_{K_{\epsilon_{2}}^{(\tau)}, \epsilon_{1} }\in C_{c}(\mathbb{R}) .$ Moreover, we have
	 ${\rm supp}\,\mu_{K_{\epsilon_{2}}^{(\tau)}, \epsilon_{1} }\subseteq |\tau|^{-1} ( \left[ 0,\epsilon_{1} \right] ) $ and $   \| \mu_{K_{\epsilon_{2}}^{(\tau)}, \epsilon_{1} }\|_{\infty}=1$. This implies that $$| \mu_{K_{\epsilon_{2}}^{(\tau)}, \epsilon_{1} }^{(\tau)} (x) \tau^{n}(x)| \leq \epsilon_{1}^{n} $$
	  for all $x \in \mathbb{R} $ and $n \in \mathbb{N}$, so $\| \mu_{K_{\epsilon_{2}}^{(\tau)}, \epsilon_{1} }\tau^{n} \|_{\infty} \leq  \epsilon_{1}^{n}  $ for all $n \in  \mathbb{N}.$ As $\epsilon_{1} <1 ,$ we get that $\mu_{K_{\epsilon_{2}}^{(\tau)}, \epsilon_{1} }\in \mathcal{A}_{\tau} .$ 
\end{proof}
\begin{lemma}\label{lem2}
	For each $f\in \mathcal{A}_{\tau}$, there is a sequence $(g_n)$ in the set$ \{\mu_{K_{\epsilon_{2}}^{(\tau)}, \epsilon_{1} }:\,\, 0<\epsilon_{2}<\epsilon_{1}<1,\,\,K_{\epsilon_{2}}^{(\tau)} \text{ is  compact}\}$ such that $fg_n\rightarrow f$ in $\mathcal{A}_{\tau}$.
\end{lemma}
\begin{proof}
	Let $\delta>0$ and $f \in \mathcal{A}_{\tau}$ with $f \neq 0$. Then there exists an $N \in \mathbb{N}$ such that 
	$$ \sum_{n =N+1}^{\infty} \| f \tau^{n} \|_{\infty}  < \dfrac{\delta}{2}.$$
	Moreover, by Lemma \ref{lem1}, $ | f |^{-1}  ((0, \infty)) \cap | \tau |^{-1} (\left[ 1,\infty \right) )= \varnothing$. Put $\epsilon:=\dfrac{\delta}{2N} .$ Since $ f \in C_{0}(\mathbb{R})$, the set 
	$$K_\epsilon:=\{x\in\mathbb{R}:\, |f(x)|\geq \epsilon\}$$
	is compact. So, for all $x \in K_{\epsilon},$ we have
	$$ \epsilon  \sum_{n =0}^{\infty} |  \tau(x) |^{n} \leq
	\sum_{n =0}^{\infty} | f(x) \tau(x)^{n} |  \leq
	\sum_{n =0}^{\infty} \| f \tau^{n} \|_{\infty}= \| f \|_{\tau}. $$
	Hence  
	$$\sum_{n =0}^{\infty} | \tau(x) |^{n} < \dfrac{\| f \|_{\tau}}{\epsilon}  \quad\text{ for all } x \in K_{\epsilon} $$
	which gives that
	$$ | \tau (x) | < 1 -\dfrac{\epsilon}{\| f \|_{\tau}} \quad\text{ for all } x \in K_{\epsilon}  . $$
	Setting
	$$\epsilon_{2}=1-\dfrac{\epsilon}{\| f \|_{\tau}},$$ we have
	$K_{\epsilon} \subseteq | \tau |^{-1} (\left[ 0,\epsilon_{2} \right] )$.
	From now on we will denote $K_{\epsilon}$ by $K_{\epsilon_{2}}^{\tau}$. Choose an $\epsilon_{1} \in (\epsilon_{2} , 1)$ and consider the corresponding function $\mu_{K_{\epsilon_{2}}^{(\tau)}, \epsilon_{1}}\in \mathcal{A}_{\tau}$. Since 
	$\mu_{ {K_{ \epsilon_{2} }^{(\tau)} , \epsilon_{1}}}=1$ on $K_{\epsilon_{2}}^{(\tau)}$, we have 
		$( f \mu_{ {K_{ \epsilon_{2} }^{(\tau)} , \epsilon_{1}} }) (x) =f(x)  $ for all $x \in K_{ \epsilon_{2} }^{(\tau)}.$ 
	Also, for each $x \in \mathbb{R} \setminus K_{ \epsilon_{2} }^{(\tau)}$, 
	$$| f(x) (\mu_{ {K_{ \epsilon_{2} }^{(\tau)} , \epsilon_{1}} }^{(\tau)} -1)(x)| \text{ } |\tau(x)|^{n} \leq \dfrac{\delta}{2N}    $$
	since $| f(x)| \leq \dfrac{\delta}{2N}$,\,\,\,$\mu_{ {K_{ \epsilon_{2} }^{(\tau)} , \epsilon_{1}} }: \mathbb{R} \rightarrow \left[ 0,1 \right] $ by the construction from Urisohn's Lemma (which gives $\| \mu_{ {K_{ \epsilon_{2} }^{(\tau)} , \epsilon_{1}} }-1 \|_{\infty}=1 ) $,  and in addition 
	$| f |^{-1}((0,\infty)) \cap | \tau |^{-1}(\left[ 1,\infty \right))= \varnothing $ by Lemma \ref{lem1}. Then, we get 
	$$\| f(\mu_{ {K_{ \epsilon_{2} }^{(\tau)}, \epsilon_{1}} }-1 )\,\tau^{n} \|_{\infty}   \leq\min \{ \dfrac{\delta}{2N}, \text{ } \|f \tau^{n} \|_{\infty} \}   \text{ for all }n \in \mathbb{N} .$$
	Hence, we deduce 
	$$ \sum_{n =1}^{\infty}  
	\| f(\mu_{ {K_{ \epsilon_{2} }^{(\tau)} , \epsilon_{1}} }-1 )   \tau^{n} \|_{\infty}  \leq        
	\sum_{n =1}^{N}  \dfrac{\delta}{2N} + \sum_{n =N+1}^{\infty} \|f \tau^{n} \|_{\infty} < \dfrac{\delta}{2} +\dfrac{\delta}{2}= \delta,$$
	and the proof is complete.
\end{proof}
\begin{corollary} \label{cor1}
	The set 
	$$\{ f \in \mathcal{A}_{\tau} \cap C_{c}(\mathbb{R}):\,\,{\rm supp} f \subseteq | \tau |^{-1} ( \left[ 0,\epsilon \right] ) \text{ for some } \epsilon\in(0,1)\}$$
	 is dense in $\mathcal{A}_{\tau} .$ 
\end{corollary}	

\begin{proof}
	Let $ f \in \mathcal{A}_{\tau}$ and  $\delta>0$. By Lemma \ref{lem2} there exists an $\mu_{K_{\epsilon_{2}}^{(\tau)}, \epsilon_{1} }\in \mathcal{A}_{\tau} $ such that $ \| f	\mu_{K_{\epsilon_{2}}^{(\tau)}, \epsilon_{1} }^{(\tau)} - f \|_{\tau} < \delta.$ Now, we have that  
	$${\rm supp}f	\mu_{K_{\epsilon_{2}}^{(\tau)}, \epsilon_{1} }\subseteq {\rm supp} \text{ } \mu_{K_{\epsilon_{2}}^{(\tau)}, \epsilon_{1} }\subseteq | \tau |^{-1} ( \left[ 0,\epsilon_{1} \right]),$$
	and $\epsilon_{1} \in (0,1)$. 
	Moreover,  ${\rm supp} \mu_{K_{\epsilon_{2}}^{(\tau)}, \epsilon_{1} }$ is compact, hence ${\rm supp}f \mu_{K_{\epsilon_{2}}^{(\tau)}, \epsilon_{1} }$ is compact as well. 
\end{proof}	

\begin{lemma} \label{l26}
	For each $k \in \mathbb{Z}$, $ f \circ \alpha \in \mathcal{A}_{\tau \circ \alpha^{k+1}} $ if and only if $ f \in \mathcal{A}_{\tau \circ \alpha^{k}} $. Moreover, in this case, $\| f \circ \alpha \|_{\tau \circ \alpha^{k+1}}= \| f \|_{\tau \circ \alpha^{k}}$. 
\end{lemma}
\begin{proof}
	Note that for each $f\in \mathcal{A}_{\tau\circ \alpha^k}$, 
	\begin{align*}
	\sum_{n=0}^\infty\|(f\circ\alpha)(\tau\circ \alpha^{k+1})^n\|_{\infty}&= \sum_{n=0}^\infty\|[(f\circ\alpha)(\tau\circ \alpha^{k+1})^n]\circ\alpha^{-1}\|_{\infty}\\
	&= \sum_{n=0}^\infty\|f(\tau\circ \alpha^{k})^n\|_{\infty} =\| f \|_{\tau \circ \alpha^{k}}.
	\end{align*}	
\end{proof}

Define
$${\rm c}_0^{\mathcal{A},\tau}:=\Big\{s:=(s_k)\in\prod_{k\in\mathbb{Z}}\mathcal{A}_{\tau\circ \alpha^k}:\,\lim_{k\rightarrow\infty}\|s_k\|_{\tau\circ\alpha^k}=\lim_{k\rightarrow\infty}\|s_{-k}\|_{\tau\circ\alpha^{-k}}=0\Big\}.$$
We equip ${\rm c_0}_{\mathcal{A},\tau}$ with the norm 
$$\|s\|_0:=\sup_{k\in\mathbb{Z}}\|s_k\|_{\tau\circ\alpha^k}.$$

\begin{lemma}
	The operator ${T}_{\alpha,W}$ is a self-mapping on ${\rm c}_0^{\mathcal{A},\tau}$.
\end{lemma}
\begin{proof}
	Let $s \in {\rm c}_0^{\mathcal{A},\tau}$. Just observe that for each $j\in\mathbb{Z}$,
	\begin{align*}
	\| ({T}_{\alpha,W}(s))_{j+1}\|_{\tau\circ\alpha^{j+1}}&\leq \|w_{j+1} \|_\infty\,\sum_{k=0}^\infty\, \|(s_j\circ\alpha)\cdot(\tau\circ\alpha^{j+1})^k\|_{\infty}\\
	&=\|w_{j+1} \|_\infty\,\sum_{k=0}^\infty\, \|(s_j\cdot(\tau\circ\alpha^{j})^k)\circ\alpha\|_{\infty}\\
	&=\|w_{j+1} \|_\infty\,\sum_{k=0}^\infty\, \|(s_j\cdot(\tau\circ\alpha^{j})^k\|_{\infty}\\
	&=\|w_{j+1} \|_\infty\,\|s_j\|_{\tau\circ\alpha^{j}} \leq \sup_{k}\|w_k\|_\infty\,\,\, \|s_j\|_{\tau\circ\alpha^{j}}.
	\end{align*}
\end{proof}
\begin{theorem}
	The followings are equivalent:
	\begin{enumerate}
	\item $T_{\alpha,W} $ is hypercyclic on ${\rm c}_0^{\mathcal{A},\tau}$.
	\item For every $J \in \mathbb{N} $ and finite collection  
	$$\{ K_{\epsilon_{J}}^{(\tau \circ \alpha^{-J})}, K_{\epsilon_{-J+1}}^{(\tau \circ \alpha^{-J+1})}, \ldots ,K_{\epsilon_{0}}^{\tau } , \ldots , K_{\epsilon_{J}}^{(\tau \circ \alpha^{J})} \} $$ 
	of compact subsets of $\mathbb{R}$, there exists an increasing sequence $(r_{k})\subseteq \mathbb{N} $
	such that 
	$$\lim_{k\rightarrow\infty} (\sup_{t \in K_{\epsilon_{j}}^{(\tau \circ \alpha^{j})}}  |  \prod_{i=1}^{r_{k}} ( w_{i+j} \circ \alpha^{-i}) (t) |)=
	\lim_{k\rightarrow\infty} (\sup_{t \in K_{\epsilon_{j}}^{(\tau \circ \alpha^{j})}}  |  \prod_{i=1}^{r_{k}} ( w_{j+1-i} \circ \alpha^{-i})^{-1} (t) | ) =0  $$ for all $j \in I_J$.
	\end{enumerate}
\end{theorem}
\begin{proof}
	The method is almost analogue of the proof of Theorem \ref{thm1}. 
	
	 $(1)\Rightarrow(2)$: Let $J \in \mathbb{N} $ and   
	 $$\{ K_{\epsilon_{J}}^{(\tau \circ \alpha^{-J})}, K_{\epsilon_{-J+1}}^{(\tau \circ \alpha^{-J+1})}, \ldots ,K_{\epsilon_{0}}^{\tau } , \ldots , K_{\epsilon_{J}}^{(\tau \circ \alpha^{J})} \} $$ 
	 be a collection of compact subsets of $\mathbb{R}$. For each $j \in I_J$, fix numbers $ 0< \epsilon_{j} < \eta_{j} <1 $ and put 
		$$\mu:=(\ldots,0, \ldots, 0, 
	\mu_{K_{\epsilon_{-J}}^{(\tau \circ \alpha^{-J})}, \eta_{-J} },
	\mu_{K_{\epsilon_{-J+1}}^{(\tau \circ \alpha^{-J+1})}, \eta_{-J+1}}, \ldots ,$$
	$$ 
	\qquad\qquad\qquad\qquad\mu_{K_{\epsilon_{0}}^{(\tau )}, \eta_{0} }, \ldots,
	\mu_{K_{\epsilon_{J-1}}^{(\tau \circ \alpha^{J-1})}, \eta_{J-1} },
	\mu_{K_{\epsilon_{J}}^{(\tau \circ \alpha^{J})}, \eta_{J} }, 0, \ldots,0, \ldots).$$ 
	Then $\mu \in {\rm c}_0^{\mathcal{A},\tau}$. By letting $\mu $ play the role of $z$ in the proof of the part $(1) \Rightarrow (2) $ in the proof of Theorem \ref{thm1} we can proceed in the same way as in that proof. It suffices just to observe that for all $v \in {\rm c}_0^{\mathcal{A},\tau}$ and $j \in \mathbb{Z} $ we have that $\| v \|_{0} \geq \| v_{j} \|_{\tau \circ \alpha^{J}} \geq \| v_{j} \|_{\infty}  $ (where $\| \cdot \|_{0} $ is the norm on ${\rm c}_0^{\mathcal{A},\tau}$, recall the definition of ${\rm c}_0^{\mathcal{A},\tau}$ and $\| \cdot \|_{0} $ ). Next, we prove $(2) \Rightarrow(1)$:\\
	Put 
	$$\mathcal{F}_{{\rm c}_0^{\mathcal{A},\tau}}:=\Big\{ y \in {\rm c}_0^{\mathcal{A},\tau}:\, \exists J \in \mathbb{N} \text{ such that } y_{j}=0 \text{ } \forall j\in\mathbb{Z}\setminus I_J,\,{\rm supp}y_{j} \text{ is compact for all }$$	
	 $$j \in I_J,\,  {\rm supp}y_{j} \subset |\tau \circ \alpha^{j}  |^{-1} ( \left[ 0,\epsilon_{j} \right])\, \,\forall j \in I_J \text{ and some } \epsilon_{-J}, \ldots , \epsilon_{J} \in (0,1)\Big \}.$$
	By Corollary \ref{cor1}, $\mathcal{F}_{{\rm c}_0^{\mathcal{A},\tau}} $ is dense in $ {\rm c}_0^{\mathcal{A},\tau}$. Let $y \in \mathcal{F}_{{\rm c}_0^{\mathcal{A},\tau}}$. Then, there exists some $J \in \mathbb{N} $ such that $y_{j}=0 $ whenever $j \notin I_J$. Set $K_{\epsilon_{j}}^{(\tau \circ \alpha^{j})}:={\rm supp}\,y_{j}$  for $j \in I_J$ where $ \epsilon_{-J}, \ldots , \epsilon_{J} \in (0,1)$ are such that 
	$$K_{\epsilon_{j}}^{(\tau \circ \alpha^{j})} \subseteq |\tau \circ \alpha^{j}| \text{ }  ( \left[ 0,\epsilon_{j} \right] ) \quad \forall j \in I_J.$$ 
	Then $K_{\epsilon_{j}}^{(\tau \circ \alpha^{j})} $ is compact for all $j \in I_J$ as $y \in \mathcal{F}_{{\rm c}_0^{\mathcal{A},\tau}}$. Hence, there exists an increasing sequence $(r_{k})\subseteq \mathbb{N} $ such that  for all $j \in I_J$ we have
	$$\lim_{k\rightarrow\infty} (\sup_{t \in K_{\epsilon_{j}}^{(\tau \circ \alpha^{j})}}  |  \prod_{i=1}^{r_{k}} ( w_{i+j} \circ \alpha^{-i}) (t) |=
	\lim_{k\rightarrow\infty} (\sup_{t \in K_{\epsilon_{j}}^{(\tau \circ \alpha^{j})}}  |  \prod_{i=1}^{r_{k}} ( w_{j+1-i} \circ \alpha^{-i})^{-1} (t) | ) =0 . $$ 
	By Lemma \ref{l26}, for all $
	j \in I_J$ and $k \in \mathbb{N} $  we get 
	\begin{align*}	
	\| (T_{\alpha,W}^{r_{k}}(y))_{j+r_{k}} \|_{(\tau \circ \alpha^{(j+r_{k})})}&=
	\| \prod_{i=1}^{r_{k}} ( w_{j+i} \circ \alpha^{r_{k}-i}) ( y_{j}\circ \alpha^{r_{k}}) \|_{(\tau \circ \alpha^{(j+r_{k})})}\\
	&= \| \prod_{i=1}^{r_{k}} ( w_{j+i} \circ \alpha^{-i})  y_{j} \|_{(\tau \circ \alpha^{j})}.
	\end{align*}
	
	Now, we have 
	\begin{align*}
	\| \prod_{i=1}^{r_{k}} ( w_{j+i} \circ \alpha^{-i})  y_{j} \|_{(\tau \circ \alpha^{j})}
	&=\sum_{n =1}^{\infty}  \| \prod_{i=1}^{r_{k}} ( w_{j+i} \circ \alpha^{-i})  y_{j} (\tau \circ \alpha^{j})^{n}\|_{\infty}\\	
	&\leq \sum_{n =1}^{\infty} (\sup_{t \in K_{\epsilon_{j}}^{(\tau \circ \alpha^{j})}}  |  \prod_{i=1}^{r_{k}} ( w_{j+i} \circ \alpha^{-i}) (t) |) \|   y_{j} (\tau \circ \alpha^{j})^{n}     \|_{\infty}\\
	&= (\sup_{t \in K_{\epsilon_{j}}^{(\tau \circ \alpha^{j})}}  |  \prod_{i=1}^{r_{k}} ( w_{j+i} \circ \alpha^{-i}) (t) | \text{ } )   \text{ }    \sum_{n =1}^{\infty} \|  y_{j} (\tau \circ \alpha^{j})^{n}     \|_{\infty}\\
	&= (\sup_{t \in K_{\epsilon_{j}}^{(\tau \circ \alpha^{j})}}  |  \prod_{i=1}^{r_{k}} ( w_{j+i} \circ \alpha^{-i}) (t) | \text{ } )   \text{ }     \|  y_{j}     \|_{ (\tau \circ \alpha^{j}) }.
	\end{align*}
	Indeed, this holds because ${\rm supp} y_{j}=K_{\epsilon_{j}}^{(\tau \circ \alpha^{j})}$. Since 
	$$\lim_{k \rightarrow \infty} (  \sup_{t \in K_{\epsilon_{j}}^{(\tau \circ \alpha^{j})}}  |  \prod_{i=1}^{r_{k}} ( w_{j+i} \circ \alpha^{-i}) (t) | \text{ } ) =0 ,$$ 
	we get that 
	$$\lim_{k \rightarrow \infty} \|(   T_{\alpha ,W}^{r_k} (y)  )_{j+r_{k}} \|_{( \tau \circ \alpha^{j+r_{k}} )}=0 .$$ 
	Next, for all $ j \in I_J$ and $k \in \mathbb{N} $ we have
	\begin{align*}
	\|(   S_{\alpha ,W}^{r_k} (y)  )_{( \tau \circ \alpha^{j-r_{k}} )}
	&= \| \prod_{i=1}^{r_{k}} ( w_{j+1-i} \circ \alpha^{i-r_{k}-1})^{-1} ( y_{j} \circ \alpha^{-r_{k}}  )  \|_{(\tau \circ \alpha^{j-r_{k}})}\\
&= \| \prod_{i=1}^{r_{k}} ( w_{j+1-i} \circ \alpha^{i-1})^{-1} ( y_{j} )  \|_{(\tau \circ \alpha^{j})}\\
	&\leq  \sup_{t \in K_{\epsilon_{j}}^{(\tau \circ \alpha^{j})}} ( |  \prod_{i=1}^{r_{k}} ( w_{j+1-i} \circ \alpha^{i-1})^{-1} (t) | \text{ } )   \text{ }     \|  y_{j}     \|_{ (\tau \circ \alpha^{j}) }. 
	\end{align*}
	Hence 
	$$\lim_{k \rightarrow \infty} \|(   S_{\alpha ,W}^{r_k} (y)  )_{j-r_{k}} \|_{( \tau \circ \alpha^{j-r_{k}} )}=0 .$$ 
	Moreover, $(   T_{\alpha ,W}^{r_k} (y)  )_{j+r_{k}} =  (S_{\alpha ,W}^{r_k} (y)  )_{j-r_{k}}=0$ whenever $j \notin I_J$. Since the set $I_J$ is finite, it follows that 
	$$ \lim_{k \rightarrow \infty} \| T_{\alpha ,W}^{r_k} (y) \|_{0}= \lim_{k \rightarrow \infty} \| S_{\alpha ,W}^{r_k} (y) \|_{0} =0.$$ 
	Thus $ T_{\alpha ,W}^{r_k} $ and $ S_{\alpha ,W}^{r_k} $ converge pointwise on a dense subset of ${\rm c}_0^{\mathcal{A},\tau}$, hence they are hypercyclic on ${\rm c}_0^{\mathcal{A},\tau}$.
\end{proof}
\begin{corollary}
	If $T_{\alpha,W}$ is hypercyclic on $l_{2}(\mathcal{A})$, then 	$T_{\alpha,W}$ is  hypercyclic on ${\rm c}_0^{\mathcal{A},\tau}$. 
\end{corollary}	
\begin{lemma}
	Let $\tau \in C_{b}(\mathbb{R})$ and assume that $\tau \circ \alpha = \tau.$ Then $T_{\alpha , w}$ is a bounded linear self-mapping on $\mathcal{A}_{\tau}.$  
\end{lemma}
\begin{proof}
	For every $f \in \mathcal{A}_{\tau} ,$ we have that
	\begin{align*}  
	\sum_{k=0}^{\infty} \| w (f \circ \alpha) \tau^{k} \|_{\infty}&= \sum_{k=0}^{\infty} \| ((w \circ \alpha^{-1})f(\tau \circ \alpha^{-1})^{k}) \circ \alpha \|_{\infty}\\
	&=\sum_{k=0}^{\infty} \| ((w \circ \alpha^{-1})f\tau^{k} \|_{\infty}\\
	&\leq  \sup_{t \in \mathbb{R}} \vert  w(t) \vert \sum_{k=0}^{\infty} \| f \tau^{k} \|_{\infty}=\sup_{t \in \mathbb{R}} \vert w(t) \vert \| f \|_{\infty}.
	\end{align*}
\end{proof}
\begin{theorem}
	The followings are equivalent:
	\begin{enumerate}
	\item  $T_{\alpha , w}$ is hypercyclic on $\mathcal{A_{\tau}}.$
	\item  For every $\epsilon \in (0,1) $ and compact subset $K_{\epsilon}^{(\tau)} \subseteq \vert \tau \vert^{-1} ([0,\epsilon]),$ there exists a strictly increasing sequence $(n_{k}) $ in $\mathbb{N}$ such that 
	$$\lim_{k \rightarrow \infty } ( \sup_{t \in K_{\epsilon}^{(\tau)}} \vert   \prod_{j=0}^{n_{k}-1} (w \circ \alpha^{j-n_{k}} ) (t)  \vert )=\lim_{k \rightarrow \infty } ( \sup_{t \in K_{\epsilon}^{(\tau)}} \vert   \prod_{j=0}^{n_{k}-1} (w \circ \alpha^{j} )^{-1} (t)  \vert )=0.$$
	\end{enumerate}
\end{theorem}
\begin{proof}
	Assume that $(1)$ holds. Given $\epsilon_{2} \in (0,1) $ and a compact subset $ K_{\epsilon_{2}}^{(\tau)} $ of $ \vert  \tau \vert^{-1}([0,\epsilon_{2}])$,  consider the function  $\mu_{K_{\epsilon_{2}}^{(\tau)}, \epsilon_{1}}$ for some $ \epsilon_{1} \in (\epsilon_{2},1)$. Then, $\mu_{K_{\epsilon_{2}}^{(\tau)}, \epsilon_{1}}=1 $ on $K_{\epsilon_{2}}^{(\tau)}$, and ${\rm supp}\mu_{K_{\epsilon_{2}}^{(\tau)}, \epsilon_{1}}$  is 	compact. Moreover, as observed earlier, $\| \cdot \|_{\tau} \geq \| \cdot \|_{\infty}$. Hence, we can proceed in exactly the same way as in the proof of $(1)\Rightarrow (2)$ of Theorem \ref{thm1}.
	
	Next, assume that $(2)$ holds. Let $ f \in \mathcal{A_{\tau}} $ such that support of $f$ is compact and ${\rm supp}f \subseteq \vert \tau \vert^{-1}  ([0,\epsilon]) $ for some $\epsilon \in (0,1)$.  Set $K_{\epsilon}^{(\tau)}:={\rm supp} f.$ Then, 
	\begin{align*}
	\|  T_{\alpha,w}^{n_{k}} (f) \|_{\tau}&=\sum_{r=0}^{\infty} \|  ( \prod_{j=0}^{n_{k}-1} (w \circ \alpha^{j} ) ) \text{ } (f \circ \alpha^{n_{k}}) \tau^{r} \|_{\infty}\\
	&= \sum_{r=0}^{\infty} \|  ( \prod_{j=0}^{n_{k}-1} (w \circ \alpha^{j-n_{k}} ) ) \text{ } f \text{ } (\tau \circ \alpha^{-n_{k}})^{r} \|_{\infty}\\
	&=\sum_{r=0}^{\infty} \|  ( \prod_{j=0}^{n_{k}-1} (w \circ \alpha^{j-n_{k}} ) ) \text{ } f \text{ } \tau^{r} \|_{\infty}\\
	&\leq \sup_{t \in K_{\epsilon}^{(\tau)}} \vert \prod_{j=0}^{n_{k}-1} (w \circ \alpha^{j-n_{k}} )  \text{ }   (t) \vert   \sum_{r=0}^{\infty} \| f \tau^{r}  \|_{\infty}\\
	&\leq \sup_{t \in K_{\epsilon}^{(\tau)}} \vert  \prod_{j=0}^{n_{k}-1} (w \circ \alpha^{j-n_{k}} ) \text{ }   (t) \vert \text{ }  \| f  \|_\tau.
	\end{align*}
	Similarly, we have 
	\begin{align*}
	\| S_{\alpha,w}^{n_{k}} (f) \|_{\tau}&=\sum_{r=0}^{\infty} \|  ( \prod_{j=1}^{n_{k}} (w \circ \alpha^{-j} )^{-1} ) \text{ } (f \circ \alpha^{-n_{k}}) \tau^{r} \|_{\infty}\\
	&= \sum_{r=0}^{\infty} \|   \prod_{j=1}^{n_{k}} (w \circ \alpha^{n_{k}-j} )^{-1}  \text{ } f \text{ } \tau^{r} \|_{\infty}\\
	& =  \sum_{r=0}^{\infty} \|  \prod_{j=0}^{n_{k}-1} (w \circ \alpha^{j} )^{-1}  \text{ } f \text{ } \tau^{r} \|_{\infty}\\	
	&\leq \sup_{t \in K_{\epsilon}^{(\tau)}} \vert \prod_{j=0}^{n_{k}-1} (w \circ \alpha^{j})^{-1}  \text{ }   (t) \vert   \| f \|_\tau .
	\end{align*}
	Since by Corollary \ref{cor1}, the set 
	$$\{ f \in \mathcal{A_{\tau}} \cap C_{c}(\mathbb{R}):\,\,{\rm supp}f \subseteq \vert \tau \vert^{-1} ([0,\epsilon]) \text{ for some } \epsilon \in (0,1)\}$$
	 is dense in $\mathcal{A_{\tau}} $, it follows that $ T_{\alpha,w}$ is topologically transitive, and so hypercyclic on $\mathcal{A_{\tau}} .$
\end{proof}
\begin{corollary}
	If $T_{\alpha,w}$ is hypercyclic on $\mathcal{A}$, then $T_{\alpha,w} $ is hypercyclic on $\mathcal{A_{\tau}}$.
\end{corollary}

\vspace{.1in}

\end{document}